\documentclass{article}
\usepackage[utf8]{inputenc}
\usepackage{srcltx}
\usepackage{amssymb,amsmath,amsfonts,amsthm}
\usepackage{cite}
\usepackage{verbatim}
\usepackage{hyperref}
\usepackage[T2A]{fontenc}
\usepackage{amsmath}
\usepackage{longtable,float}
\usepackage{tikz}
\usepackage{amssymb}
\usepackage{amsfonts}
\usepackage{textcomp}
\usepackage{enumerate}

\newtheorem{theorem}{Theorem}[section]
\newtheorem{lemma}[theorem]{Lemma}
\newtheorem{prop}[theorem]{Proposition}
\newtheorem{remark}[theorem]{Remark}

\newcommand{\F}{\mathbb{F}}

\numberwithin{equation}{section}

\DeclareMathOperator{\SL}{SL}
 \DeclareMathOperator{\GO}{GO}
\DeclareMathOperator{\SO}{SO} 
\DeclareMathOperator{\PSL}{PSL} 
  
\DeclareMathOperator{\PSU}{PSU}  \DeclareMathOperator{\SU}{SU}
     \DeclareMathOperator{\Sym}{Sym}
 \DeclareMathOperator{\bd}{bd}
\DeclareMathOperator{\diag}{diag} 
 \DeclareMathOperator{\Sl}{Sl}

\title{\vspace{-1cm} \hfill{%\normalsize УДК 512.542
	} \\
On splitting of the normalizers of maximal tori in finite groups of Lie type}

\author{A.A. Galt, A.M. Staroletov\footnote{The work is supported by RAS Fundamental Research Program (project  FWNF-2022-0002).}}
\date{\vspace{-5ex}}
\begin{document}
%\subjclass[2000]{Primary 54C40, 14E20; Secondary 46E25, 20C20}
\newcommand{\Addresses}{{% additional braces for segregating \footnotesize
  \bigskip
  \footnotesize

 Alexey~Galt, \textsc{Sobolev Institute of Mathematics, Novosibirsk, Russia;}\par\nopagebreak
\textsc{Novosibirsk State University, Novosibirsk, Russia;}\par\nopagebreak
\textit{E-mail address: } \texttt{galt84@gmail.com}

\medskip

Alexey~Staroletov, \textsc{Sobolev Institute of Mathematics, Novosibirsk, Russia;}\par\nopagebreak
\textsc{Novosibirsk State University, Novosibirsk, Russia;}\par\nopagebreak
\textit{E-mail address: } \texttt{staroletov@math.nsc.ru}
}}

\maketitle
\begin{flushright}
{\it Dedicated to V. D. Mazurov on the occasion of his 80th birthday}
\end{flushright}

\begin{abstract}
Let $G$ be a finite group of Lie type and $T$ a maximal torus of $G$. In this paper we complete the study of the question of the existence of a complement for the torus $T$ in its algebraic normalizer $N(G,T)$. It is proved that every maximal torus of the group $G\in\{G_2(q), {}^2G_2(q), {}^3D_4(q)\}$ has a complement in its algebraic normalizer. The remaining twisted classical groups ${}^2A_n(q)$ and ${}^2D_n(q)$ are also considered.
\end{abstract}

{\bf Keywords:} finite group of Lie type, twisted group of Lie type, Weyl group, maximal torus, algebraic normalizer

\section{Introduction}

The splitting problem for the normalizer of a maximal torus was first formulated by J.~Tits~\cite{Tits}. Let $\overline{G}$ be a simple connected linear algebraic group over the algebraic closure $\overline{\F}_p$ of the prime field of characteristic $p$. Let $\sigma$ be a Steinberg endomorphism and $\overline{T}$ a maximal $\sigma$-invariant torus of $\overline{G}$. It is well known that all maximal tori are conjugate in $\overline{G}$~\cite[Corollary 6.5]{MalleTester} and  the quotient group $N_{\overline{G}}(\overline{T})/\overline{T}$ is isomorphic to the Weyl group $W$ of $\overline{G}$. The natural question is for which groups $\overline{G}$ the normalizer $N_{\overline{G}}(\overline{T})$ splits over $\overline{T}$. The answer was obtained independently in~\cite{AdamsHe} and in a series of papers~\cite{Galt1,Galt2,Galt3,Galt4}. The same problem for the Lie groups was solved in~\cite{LieGroups}.

A similar question can be formulated for finite groups of Lie type. Let $G$ be a finite group of Lie type. Suppose that $T=\overline{T}\cap G$ is a maximal torus of $G$ and $N(G,T)=N_{\overline{G}}(\overline{T})\cap G$ is its algebraic normalizer in $G$. It is known that in the case of finite groups the maximal tori need not be conjugate in $G$. The general problem is to describe groups $G$ and their maximal tori $T$ such that $N(G,T)$ splits over $T$. This problem has been solved for groups of Lie types $A_n$, $B_n$, $C_n$, $D_n$, $E_6$, $E_7$, $E_8$, and $F_4$ in~\cite{Galt2,Galt3,Galt4, GS,GS2,GS3}. Moreover, in the case of $F_4(q)$ supplements of minimal order for the maximal torus were found if there is no complement.

J. Adams and X. He in~\cite{AdamsHe} considered a related question: what is the minimal  order of a lift for an element $w\in W$ to $N_{\overline{G}}(\overline{T})$?
It is easy to see that if the order of $w$ is $d$, then the minimal order of a lift for $w$ is either $d$ or $2d$. Clearly, if $N_{\overline{G}}(\overline{T})$ splits over $\overline{T}$, then the minimal order is $d$.
For algebraic groups $\overline{G}$ of Lie type $E_6$, $E_7$, $E_8$, or $F_4$, the normalizer of the torus  does not split. Nevertheless, in~\cite{GS,GS2} it is proved that in the case of finite groups $E_6(q)$, $E_7(q)$, $E_8(q)$ the minimal order of a lift is always equal to $d$. In particular, the minimal order of a lift is $d$ in the corresponding algebraic groups.

For finite groups of Lie type $F_4$ in~\cite{AdamsHe}, minimal orders of lifts are found for the elements belonging to the so-called regular or elliptic conjugacy classes of $W$. In particular, J.\,Adams and X.\,He showed that there exists an elliptic element of order four for which any of its lift has order eight. In~\cite{GS3}, minimal orders of lifts are found for all elements of the Weyl group in $F_4(q)$.

This paper completes the study of the above question for the finite groups of Lie type.
The first result is devoted to exceptional groups of Lie type.

\begin{theorem}\label{th}
Let $G\in\{G_2(q), {}^2G_2(q), {}^3D_4(q)\}$. If $T$ is a maximal torus of $G$ and $N$ is its algebraic normalizer, then $N$ splits over $T$.
\end{theorem}

Among the classical groups it remains to consider the twisted groups ${}^2A_n(q)$ and ${}^2D_n(q)$.
In~\cite{Galt3}, the splitting problem for maximal tori is solved for groups $A_n(q)$. For twisted groups ${}^2A_n(q)$, the result can be obtained as follows.
The structure of maximal tori in the special unitary group $\SU_n(q)$ is described in~\cite[\S2]{ButGre} and is obtained from the structure of the corresponding maximal tori in the special linear group $\SL_n(q)$ by replacing $q$ with $ -q$. Repeating results from~\cite{Galt3} with $q$ replaced by $-q$, we obtain results for the groups $\SU_n(q)$ and $\PSU_n(q)$. 

For the sake of completeness, we formulate the corresponding results. The maximal power of $2$ dividing a positive integer $n$ is denoted by $(n)_2$. We  use the notation $\SL_n^\varepsilon(q)$, where $\varepsilon=\pm$, setting $\SL_n^+(q)=\SL_n(q)$ and $\SL_n ^-(q)=\SU_n(q)$. Recall that in these cases the Weyl group is isomorphic to the symmetric group $\Sym_n$ of degree~$n$. 
It is well known that conjugacy classes in $\Sym_n$ are uniquely determined by the cycle structure of a representative, and therefore are in one-to-one correspondence with partitions of $n$, that is, representations of $n$ as a sum of positive integers.
If $n=n_1+n_2+\ldots+n_m$ is a partition, then the expression $(n_1)(n_2)\ldots(n_m)$ is called the cycle type of the corresponding conjugacy class. We identify partitions that differ by a permutation of terms.
Given a partition $n=n_1+n_2+\ldots+n_m$, we denote by $b_i$ the number of its terms equal to $n_i$ and define $a_i=n_ib_i$ for $1\leqslant i\leqslant m$.

\begin{theorem}\label{th_SU}
Let $T$ be a maximal torus of $G=\SL_{n}^\varepsilon(q)$ corresponding to an element of the Weyl group with the cycle type $(n_1)(n_2)\ldots(n_m)$.
Then $T$ has a complement in $N(G,T)$ if and only if
$q$ is even or $a_i$ is odd for some $1\leqslant i\leqslant r$.
\end{theorem}

\noindent Now we formulate the result for simple linear and unitary groups.

\begin{theorem}\label{th_PSU}
Let $T$ be a maximal torus of $G=\SL_{n}^\varepsilon(q)$ corresponding to an element of the Weyl group with the cycle type $(n_1)(n_2)\ldots(n_m)$. Let $\widetilde{T}$ and $\widetilde{N}$ be the images of $T$ and $N(G,T)$ in $\widetilde{G}=\PSL_{n}^\varepsilon (q)$.
Then $\widetilde{T}$ has a complement in~$\widetilde{N}$ if and only if one of the following holds:
\begin{itemize}
\item[{\em (1)}] $q$ is even;
\item[{\em (2)}] $a_i$ is odd for some $1\leqslant i\leqslant r$;
\item[{\em (3)}] $(n)_2<(\varepsilon q-1)_2$;
\item[{\em (4)}] $m=4$, $n_1,n_2,n_3,n_4$ are odd;
\item[{\em (5)}] $m=3$, $n_1=n_2$ is odd, $n_3$ is even, $(n_3)_2>2$, and $(n)_2\leqslant(\varepsilon q-1)_2$;
\item[{\em (6)}] $m=3$, $n_1=n_2$ is odd, $n_3$ is even, $(n_3)_2=2$, and $(n)_2\neq(\varepsilon q-1)_2$;
\item[{\em (7)}] $m=2$, $n_1,n_2$ are odd;
\item[{\em (8)}] $m=2$, $n_1,n_2$ are even, $n_1\neq n_2$, $(n)_2<d(\varepsilon q-1)_2$, where $d=\gcd\left((\frac{n_1}{2})_2,(\frac{n_2}{2})_2,(\varepsilon q-1)_2\right)$;
\item[{\em (9)}] $m=2$, $n_1,n_2$ are even, $n_1\neq n_2$, $(n_1)_2=(n_2)_2\leqslant(\varepsilon q-1)_2,$ $(\varepsilon q-1)_2(n_1)_2\leqslant(n)_2$;
\item[{\em (10)}] $m=2$, $n_1=n_2$ is even, $(n_1)_2>2$, $(n)_2\leqslant(\varepsilon q-1)_2$;
\item[{\em (11)}] $m=2$, $n_1=n_2$ is even, $(n_1)_2=2$, $(n)_2\neq(\varepsilon q-1)_2$;
\item[{\em (12)}] $m=1$.
\end{itemize}
\end{theorem}

The structures of maximal tori and the Weyl group in the orthogonal group $\mathrm{P}\Omega_{2n}^-(q)$ are described, for example, in~\cite[\S4]{ButGre}.
Let $n=n'+n''$, $n'=n_1+\ldots+n_k$ and $n''=n_{k+1}+\ldots+n_m$ be partitions of numbers $n'$ and $n''$, respectively, where $k$ is odd. All such possible representations of the number $n$ as a sum $n=n_1+\ldots+n_k+n_{k+1}+\ldots+n_m$ are in one-to-one correspondence with the conjugacy classes of the Weyl group. As before, we identify partitions that differ by a permutation of  terms inside each of the partitions $n'$ and $n''$. If the conjugacy class corresponds to the decomposition $n=n_1+\ldots+n_k+n_{k+1}+\ldots+n_m$, then the expression $(\overline{n_1})\ldots(\overline{n_k})(n_{k +1})\ldots(n_m)$ is called the cyclic type of this class.
Given a cycle type $(\overline{n_1})\ldots(\overline{n_k})(n_{k+1})\ldots(n_m)$, we denote by $b_i$ the number of terms equal to $n_i$ in the partition $n'$ if $1\leqslant i\leqslant k$, and in the partition $n''$ if $k+1\leqslant i\leqslant m$. Define $a_i=n_ib_i$ for $1\leqslant i\leqslant m$.

\begin{theorem}\label{th_Omega}
Let $G=\mathrm{P}\Omega_{2n}^-(q)$ with $n\geqslant4$. Let $T$ be a maximal torus of $G$ corresponding to an element of the Weyl group with the cycle type $(\overline{n_1})\ldots(\overline{n_k})(n_{k+1})\ldots(n_m)$, where $k$ is odd. Then $T$ has a complement in~$N(G,T)$ if and only if one of the following holds:
\begin{itemize}
\item[{\em (1)}] $q\not\equiv3\pmod4;$
\item[{\em (2)}] $a_i$ is odd for some $1\leqslant i\leqslant m;$
\item[{\em (3)}] $k=m$, $n_i$ is even for all $1\leqslant i\leqslant k$.
\end{itemize}
\end{theorem}

We can now formulate a final theorem on splitting of the normalizer of a maximal torus in finite groups of Lie type.
\begin{theorem}\label{final}
Let $G$ be a finite simple group of Lie type and $T$ be a maximal torus of $G$. Suppose that the algebraic normalizer $N(G,T)$ splits over $T$. Then the pair $(G,T)$ is described.
\end{theorem}

The paper is organized as follows. In Section $2$, we recall the notation and basic facts used in this paper.
Theorem~\ref{th} is proved in Sections $3$, $4$, and $5$ for the groups $G_2(q),{}^2G_2(q)$, and ${}^3D_4(q)$, respectively. Section $6$ is devoted to the proof of Theorem~\ref{th_Omega}. Finally, Theorem~\ref{final} is proved in Section $7$.

\section{Notations and preliminary results}

If a group $A$ is the product of its normal subgroup $B$ and subgroup $C$, then $C$ is called a {\it supplement} to $B$ in $A$. If, in addition, $B\cap C=1$, then $C$ is called a {\it complement} to $B$ in $A$.

By $q$ we always mean a power of a prime $p$. We write $\overline{\F}_p$ for the algebraic closure of a finite field $\F_p$ of order $p$. The symmetric group
of degree $n$ is denoted by $\Sym_n$, the dihedral group of order $2n$ denoted by $D_{2n}$ and the cyclic group of order $n$ denoted by $\mathbb{Z}_n$ or $n$ for short when used in tables.

By $\overline{G}$ we denote a simple simply connected linear algebraic group over $\overline{\F}_p$ with a root system $\Phi$, and by $\Delta=\{r_1,r_2,\ldots,r_l\}$ a fundamental root system of $\Phi$.

In what follows, we will use the notation from~\cite{CarSG}, in particular, the definitions of the elements $x_r(t)$, $n_r(\lambda)$, $(r\in\Phi, t\in\overline{\F}_p, \lambda\in \overline{\F}_p^*)$.
Unlike~\cite{CarSG}, MAGMA uses the definition $h_r(\lambda) = n_r(-1)n_r(\lambda)$, which we will adhere to in this paper. According to~\cite{CarSG}, $\overline{G}$ is generated by the elements $x_r(t)$: $G=\langle x_r(t)~|~r\in\Phi,t\in \overline{\F}_p\rangle$. 
The group $\overline{T}=\langle h_r(\lambda)~|~r\in\Delta,\lambda\in \overline{\F}_p^*\rangle$ is a maximal torus in $\overline{G} $ and $\overline{N}=\langle \overline{T},n_r~|~r\in\Delta\rangle$, where $n_r=n_r(1)$, is the normalizer of $\overline{T}$ in $\overline{G}$~\cite[\S 7.1, 7.2]{CarSG}. The Weyl group $\overline{N}/\overline{T}$ is denoted by $W$ and the natural homomorphism from $\overline{N}$ onto $W$ by $\pi$. 
Similarly to~\cite[Theorem~7.2.2]{CarSG}, we have the following equalities:
\begin{center}
	$n_s n_r n_s^{-1}=n_{w_s(r)}(\eta_{s,r}),\quad \eta_{s,r}=\pm1,$
\end{center}
\begin{center}
	$n_s h_r(\lambda)n_s^{-1}=h_{w_s(r)}(\lambda).$
\end{center}
We choose the values of $\eta_{r,s}$ in a standard way (see~\cite[\S2]{GS3}).

We follow the definition of the Steinberg endomorphism from~\cite[Definition 2.1.9]{GorLySol} and denote it by $\sigma$. Further, by $\overline{T}$ we denote a maximal $\sigma$-invariant torus. Then the action of $\sigma$ on $W$ is defined in a natural way. Elements $w_1, w_2\in W$ are called {\it $\sigma$-conjugated} if $w_1=w^{-1}w_2w^{\sigma}$ for an element $w\in W$. For  $G=\overline{G}_\sigma$ and $g\in\overline{G}$ the following statements hold.
\begin{prop}{\em\cite[Propositions~3.3.1, 3.3.3]{Car}}\label{torus}.
A torus $\overline{T}^g$ is $\sigma$-invariant if and only if $g^{\sigma}g^{-1}\in\overline{N}$.
The map $\overline{T}^g\mapsto\pi(g^{\sigma}g^{-1})$ defines a bijection between the $G$-classes of $\sigma$-invariant maximal tori of the group $\overline{G}$ and the
$\sigma$-conjugacy classes of $W$.
\end{prop}

\begin{prop}{\em\cite[Lemma~1.2]{ButGre}}\label{prop2.5}.
Let $n=g^{\sigma}g^{-1}\in\overline{N}$. Then $(\overline{T}^g)_\sigma=(\overline{T}_{\sigma n})^g$, where $n$ acts on $\overline{T}$ by conjugation.
\end{prop}

\begin{prop}{\em\cite[Proposition~3.3.6]{Car}}\label{p:normalizer}.
Let $g^{\sigma}g^{-1}\in\overline{N}$ and $\pi(g^{\sigma}g^{-1})=w$. Then $$(N_{\overline{G}}({\overline{T}}^g))_{\sigma}/({\overline{T}}^g)_{\sigma}\simeq C_{W,\sigma}(w)=\{x\in W~|~x^{-1}wx^{\sigma}=w\}.$$
\end{prop}
It follows from Proposition~\ref{prop2.5} that $({\overline{T}}^g)_{\sigma}=(\overline{T}_{\sigma n})^g$ and $(N_{\overline{G}}({\overline{T}}^g))_{\sigma}=(\overline{N}^g)_{\sigma}=(\overline{N}_{\sigma n})^g$. Therefore, $({\overline{T}}^g)_{\sigma}$ has a complement in its algebraic normalizer if and only if there exists a complement to $\overline{T}_{\sigma n}$ in $\overline{N}_{\sigma n}$.

The following remark shows that the splitting of an algebraic normalizer does not depend on the choice of a torus corresponding to the element $w$.

\begin{remark}\label{r:nonsplit}
Let $n$ and $w$ be as in Propositions~\ref{prop2.5} and~\ref{p:normalizer}, respectively. Suppose that $n_1=g_1^{\sigma}g_1^{-1}\in\overline{N}$ and $\pi(n_1)=w$.
Consider two maximal tori $({\overline{T}}^g)_{\sigma}$ and $({\overline{T}}^{g_1})_{\sigma}$ of $G$ corresponding to $w$. By Proposition~\ref{torus} there exists an element $x\in G$ such that $({\overline{T}}^g)^x={\overline{T}}^{g_1}$. Then
$$(N_{\overline{G}}(\overline{T}^{g_1}))_\sigma=(N_{\overline{G}}(\overline{T}^{gx}))_\sigma
=((N_{\overline{G}}(\overline{T}^{g}))^x)_\sigma=((N_{\overline{G}}(\overline{T}^{g} ))_\sigma)^x,$$
where the latter equality follows from the fact that $x\in G=\overline{G}_\sigma$.

Thus, if maximal tori are $G$-conjugate, then their algebraic normalizers are $G$-conjugate. Therefore, we can choose a convenient for calculations representative $n_1$ with $\pi(n_1)=w$ in the proof of Theorem~\ref{th}.
\end{remark}

For brevity, for an arbitrary $r_i\in\Phi$ we write $w_i$, $h_i$, and $n_i$ instead of $w_{r_i}$, $h_{r_i}(-1)$, and $n_{r_i}$, respectively. Every element $H$ of $\overline{T}$ can be written as $H=h_{r_1}(\lambda_1)h_{r_2}(\lambda_2)\ldots h_{r_l}(\lambda_l)$ which we simplify to $(\lambda_1,\lambda_2,\ldots,\lambda_l)$.

Define $\mathcal{T}=\langle n_r~|~ r\in\Delta\rangle$ and $\mathcal{H}=\overline{T}\cap\mathcal{T}$. According to~\cite[\S4.6]{Tits}, we have $\mathcal{H}=\langle h_r~|~r\in\Delta\rangle$ and $\mathcal{T}/\mathcal{H}\simeq W$.
In particular, if $q$ is odd, then $\mathcal{H}$ is an elementary abelian $2$-group.

For the untwisted groups of Lie type, the following lemma holds, which will be used for the group $G_2(q)$.
\begin{lemma}{\em \cite[Lemma~3.1]{GS}}\label{normalizer}
Let $g\in\overline{G}$ such that $n=g^\sigma g^{-1}\in\overline{N}$. Suppose that $H\in \overline{T}$ and $u\in\mathcal{T}$. Then

(1) $Hu\in\overline{N}_{\sigma n}$ if and only if $H=H^{\sigma n}[n,u];$

(2) if $H\in \mathcal{H}$, then $Hu\in\overline{N}_{\sigma n}$ is equivalent to the equality $[n,Hu]=1$.
\end{lemma}

We need the following assertion for calculations in ${}^2G_2(q)$.

\begin{lemma}{\em \cite[Theorem~1.12.1(e)]{GS}}\label{coroots}
Let $\Delta=\{r_1,\ldots,r_l\}$ be a fundamental root system of a root system $\Phi$. For each $r\in\Phi$, we set $\check{r}=2r/(r,r)$, and write $\check{r}=\Sigma c_i\check{r_i}$. Then $h_r(t)=\Pi_{i=1}^lh_{r_i}(t^{c_i})$.
\end{lemma}

\begin{lemma}{\em \cite[Lemma~2.5]{GS2}}\label{l:ww0}
Let $w\in W$ and $w_0\in Z(W)$. Consider the maximal tori $T$ and $T_0$ corresponding to $w$ and $ww_0$, respectively.
Then $T$ has a complement in its algebraic normalizer if and only if $T_0$ has a complement in its algebraic normalizer.
\end{lemma}
Notice that we can refer to Lemma~\ref{l:ww0} in the proof of Theorem~\ref{th}. Nevertheless, if a complement to a maximal torus exists, then we construct it directly and therefore consider both cases $w$ and $ww_0$.

\section{Proof of Theorem~\ref{th} for $G_2(q)$}

In this section, we suppose that $G=G_2(q)$, where $q$ is a power of a prime~$p$. Since the case $p=2$ follows from~\cite[Remark 2.5]{GS3}, we can assume that $p$ and $q$ are odd. The Dynkin diagram of type $G_2$ has the following form:

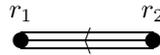
\begin{picture}(100,40)(-140,-10)
	\put(50,0){\line(1,0){50}}
	\put(50,3){\line(1,0){50}} \put(50,-3){\line(1,0){50}}
	\put(50,0){\line(1,0){50}} 
	\put(50,0){\circle*{6}} \put(100,0){\circle*{6}}
	\put(50,10){\makebox(0,0){$r_1$}}
	\put(100,10){\makebox(0,0){$r_2$}}
	\put(75,0){\makebox(0,0){$\langle$}}
\end{picture}\\

Following~\cite{Bour}, we use the following order of positive roots:
$$r_3=r_1+r_2, r_4=2r_1+r_2, r_5=3r_1+r_2, r_6=3r_1+2r_2.$$
The Weyl group $W=\langle w_1,w_2\rangle$ of $G_2(q)$ is isomorphic to the group $D_{12}$ and contains a central involution $w_0=w_1w_6=w_3w_5$. 
Since $\sigma$ acts trivially on $W$ in this case, the group $W$ contains exactly six $\sigma$-conjugacy classes and they coincide with the ordinary conjugacy classes.

To prove Theorem~\ref{th} we consider each $\sigma$-conjugacy class of maximal tori separately. As an element of $W$ corresponding to a conjugacy class of some maximal torus, we choose $w$ according to Table~\ref{tableG_2}. 
In each case, we present an element $n$ such that $\pi(n)=w$ and a complement to the torus $\overline{T}_{\sigma n}$ in $\overline{N}_{\sigma n}$.
The cyclic structure of maximal tori, given in Table~\ref{tableG_2}, can be found in~\cite[\S2]{Kantor}. The proof of the theorem uses calculations obtained with the aid of computer systems MAGMA~\cite{MAGMA,MC} and GAP~\cite{GAP}. The corresponding commands can be found in~\cite{github}. All calculations of this kind can be checked manually. For example, the case $w=1$ is considered in details in~\cite[\S7]{Galt1}.
Calculations in MAGMA show that $n_0:=h_1n_1n_6$ lies in $Z(\mathcal{T})$. In what follows, we use this fact without explanation.

\textbf{Tori 1 and 4.} In this case $w=1$ or $w=w_0$, respectively.
Moreover, we see that $C_{W}(w)\simeq\langle w_1,w_2 \rangle\simeq D_{12}$.

Let $n=1$ if $w=1$ and $n=n_0$ if $w=w_0$.
Consider elements $a=h_2n_1$ and $b=h_1n_2$. It follows from the definition of $n$ that $[n,a]=[n,b]=1$.
By Lemma~\ref{normalizer}(2), we get that $a,b\in\overline{N}_{\sigma n}$.
According to~\cite[\S7]{Galt1}, it is true that $a^2=b^2=1$ and $(ab)^6=1$. 
Therefore, we infer that $K=\langle a,b \rangle$ is a complement to $\overline{T}_{\sigma n}$ in $\overline{N}_{\sigma n}$.

\textbf{Tori 2 and 3.} In this case $w=w_2$ or $w=w_2w_0=w_4$, respectively.
Moreover, we have $C_W(w)=\langle w_2,w_4\rangle\simeq \mathbb{Z}_2\times \mathbb{Z}_2$.

Let $n=h_1n_2$ if $w=w_2$ and $n=h_1n_2n_0$ if $w=w_4$.
Consider elements  $a=h_1n_2$ and $b=h_1n_4$.
Using MAGMA, we see that $[n, a]=[n, b]=1$. By Lemma~\ref{normalizer}(2), we get that $a, b\in\overline{N}_{\sigma n}$. Now $a^2=b^2=[a,b]=1$ and hence $K=\langle a,b \rangle$ is a homomorphic image of $\mathbb{Z}_2\times \mathbb{Z}_2$. On the other hand, the image of $K$ in $W$ has order four, so $K\simeq\mathbb{Z}_2\times \mathbb{Z}_2$ and it is a complement to $\overline{T}_{\sigma n}$ in $\overline{N}_{\sigma n}$.

\textbf{Tori 5 and 6.} In this case $w=w_1w_3$ or $w=w_1w_3w_0=w_1w_5$, respectively. 
Moreover, we have $C_W(w)=\langle w_1w_3\rangle\times\langle w_0\rangle\simeq\mathbb{Z}_3\times \mathbb{Z}_2\simeq\mathbb{Z}_6$.

Let $n=n_1n_3$ if $w=w_1w_3$ and $n=n_1n_3n_0$ if $w=w_1w_5$.
Consider elements $a=n_1n_3$ and $n_0$.
Using MAGMA, we see that $[n, a]=[n, n_0]=1$. Therefore, $a, n_0\in\overline{N}_{\sigma n}$ by Lemma~\ref{normalizer}(2). Since $a^3=n_0^2=1$ and $[a,n_0]=1$, we infer that $K=\langle a,n_0\rangle$ is a complement to $\overline{T}_{\sigma n}$ in $\overline{N}_{\sigma n}$.

\begin{table}[H]
\begin{center}
\caption{Normalizers of maximal tori of $G_2(q)$\label{tableG_2}}
{\centering
\begin{tabular}{|l|l|c|l|l|l|}
    \hline
    &  $w$ & $|w|$ & $|C_W(w)|$ & $C_W(w)$ & Structure of $T$ \\ \hline
    1  & $1$ & 1 & 12 &  $D_{12}$ & $(q-1)^2$  \\
    2  & $w_2$ & 2 & 4 & $\mathbb{Z}_2\times \mathbb{Z}_2$ & $q^2-1$  \\
    3  & $w_4$ & 2 & 4 & $\mathbb{Z}_2\times \mathbb{Z}_2$ & $q^2-1$  \\
    4  & $w_1w_6$ & 2 & 12 & $D_{12}$ & $(q+1)^2$  \\
    5  & $w_1w_3$ & 3 & 6 & $\mathbb{Z}_6$ & $q^2+q+1$  \\
    6  & $w_1w_5$ & 6 & 6 & $\mathbb{Z}_6$ & $q^2-q+1$  \\
    \hline
\end{tabular}}
\end{center}
\end{table}

\section{Proof of Theorem~\ref{th} for groups ${}^2G_2(q)$}

In this section, we suppose that $G={}^2G_2(q)$, where $q=3^{2m+1}$ and $m$ is a positive integer. We follow the definition of the group ${}^2G_2(q)$ from~\cite[1.15.4, 2.2.3]{GorLySol}. In particular, $\rho$ is the nontrivial symmetry of the Dynkin diagram, $\sigma=\psi^{2m+1}$, where
\begin{center}
$\psi(x_r(t))=
\begin{cases} 
    x_{r^\rho}(t) & \text{if } r \text{ is a long root},\\
    x_{r^\rho}(t^3) & \text{if } r \text{ is a short root}.
\end{cases}$
\end{center}
As was mentioned above, the Weyl group $W=\langle w_1,w_2\rangle$ of $G_2(q)$ is isomorphic to $D_{12}$ and contains a central involution $w_0=w_1w_6=w_3w_5$. 
The $\sigma$-conjugacy classes of $W$ can be found directly by the definition. For example, the $\sigma$-conjugacy class for the identity element consists of the elements
$\{w^{-1}w^\sigma~|~w\in W\}$. Elements of the Weyl group are written in the following way: $W=\{w_k, (w_1w_2)^k~|~k=1,\ldots,6 \}$.
Since $(w_1w_2)^\sigma=w_2w_1$, we see that
$$(w_1w_2)^{-k}((w_1w_2)^k)^\sigma=(w_2w_1)^k(w_2w_1)^k=(w_2w_1)^{2k}=(w_1w_2)^{-2k},\quad w_1^{-1}w_1^\sigma=w_1w_2,$$ $$w_2^{-1}w_2^\sigma=w_2w_1=(w_1w_2)^5,\quad w_3^{-1}w_3^\sigma=w_3w_5=w_0=(w_1w_2)^3,\quad w_4^{-1}w_4^\sigma=w_4w_6=(w_1w_2)^5,$$  
$$w_5^{-1}w_5^\sigma=w_5w_3=w_0=(w_1w_2)^3,\quad w_6^{-1}w_6^\sigma=w_6w_4=w_1w_2.$$
Therefore, $\{w^{-1}w^\sigma~|~w\in W\}=\{(w_1w_2)^k~|~k=1,\ldots,6\}$. Similarly, we find the remaining $\sigma$-conjugacy classes: $\{w_1, w_2\}$, $\{w_3, w_5\}$, and $\{w_4, w_6\}$.

%Using the size of the $\sigma$-conjugacy class of $w$, we find the order of $C_{W,\sigma}(w)$ (see Table~\ref{table2G_2}).

For each representative of $\sigma$-conjugacy class, we find the structure of the corresponding maximal torus, and then construct a complement in the corresponding algebraic normalizer. The obtained results on the structure of maximal tori and their normalizers are given in Table~\ref{table2G_2}.

\textbf{Torus~1.} In this case $w=1$. 
Given an element $h_{r_1}(t_1)h_{r_2}(t_2)\in\overline{T}$, 
we find necessary and sufficient conditions ensuring that this element belongs to $\overline{T}_\sigma$.
According to~\cite[1.15.4(b)]{GorLySol}, we have $(h_r(t))^{\psi^2}=h_r(t^3)$. Hence,
$$h_{r_1}(t_1)h_{r_2}(t_2)=(h_{r_1}(t_1)h_{r_2}(t_2))^\sigma=(h_{r_1}(t_1)h_{r_2}(t_2))^{\psi^{2m}\psi}=
(h_{r_1}(t_1^{3^m})h_{r_2}(t_2^{3^m}))^\psi= h_{r_2}(t_1^{3^{m+1}})h_{r_1}(t_2^{3^m}).$$
Therefore, we find two necessary and sufficient equalities: $t_1=t_2^{3^m}$ and $t_2=t_1^{3^{m+1}}$.
Applying equivalent transformations, we get that
$$
\begin{cases}
t_1=t_2^{3^m} \\
t_2=t_1^{3^{m+1}}
\end{cases}\Leftrightarrow
\begin{cases}
t_1=t_2^{3^m} \\
t_2=t_2^{3^m\cdot 3^{m+1}}
\end{cases}\Leftrightarrow
\begin{cases}
t_1=t_2^{3^m} \\
t_2^{q-1}=1
\end{cases}.
$$
Thus, the torus is parametrized by a set $\{(z,z^{\sqrt{3q}})~|~z\in\overline{\F}_p, z^{q-1}=1 \}$ and is isomorphic to the cyclic group of order $q-1$.

Note that $C_{W,\sigma}(w)\simeq\langle w_0 \rangle\simeq\mathbb{Z}_2$. Consider the element $n=n_0=h_1n_1n_6$.
Using MAGMA, we see that $n^\sigma=n$ and $n^2=1$. Therefore, it follows from the definition of $\overline{N}_{\sigma n}$ that $n\in\overline{N}_{\sigma n}$. Hence, the group $K=\langle n \rangle$ is a complement to $\overline{T}_{\sigma n}$ in $\overline{N}_{\sigma n}$.

\textbf{Torus~2.} In this case $w=w_1$.
Using Lemma~\ref{coroots}, we get for an element
$h_{r_1}(t_1)h_{r_2}(t_2)\in\overline{T}$ that
$$\left(h_{r_1}(t_1)h_{r_2}(t_2)\right)^{\sigma n_1}=\left(h_{r_1}(t_2^{3^m})h_{r_2}(t_1^{3^{m+1}})\right)^{n_1}=
h_{-r_1}(t_2^{3^m})h_{r_1+r_2}(t_1^{3^{m+1}})=$$
$$h_{r_1}(t_2^{-3^m})h_{r_1}(t_1^{3^{m+1}})h_{r_2}(t_1^{3^{m+1}})=
h_{r_1}(t_1^{3^{m+1}}t_2^{-3^m})h_{r_2}(t_1^{3^{m+1}}).$$
Consider the following equivalent systems of equations for the parameters of $\overline{T}_{\sigma n_1}$:
$$
\begin{cases}
t_1=t_1^{3^{m+1}}t_2^{-3^m} \\
t_2=t_1^{3^{m+1}}
\end{cases}\Leftrightarrow
\begin{cases}
t_1=t_1^{3^{m+1}}t_1^{-3^{2m+1}} \\
t_2=t_1^{3^{m+1}}
\end{cases}\Leftrightarrow
\begin{cases}
t_1^{3^{2m+1}-3^{m+1}+1}=1 \\
t_2=t_1^{3^{m+1}}
\end{cases}
.
$$
Therefore, elements of the torus are parametrized by a set  $\{(z,z^{\sqrt{3q}})~|~z\in\overline{\F}_p, z^{q-\sqrt{3q}+1}=1\}$,
so $\overline{T}_{\sigma n_1}$ is isomorphic to the cyclic group of order $q-\sqrt{3q}+1$.

Note that $C_{W,\sigma}(w)\simeq\langle w_1w_2 \rangle\simeq\mathbb{Z}_6$.
Consider elements $a=n_1n_2$ and $n=n_1$. Clearly, $a^{\sigma n}=nn_2n_1n^{-1}=n_1n_2=a$, that is $a\in\overline{N}_{\sigma n}$.
As was noted in the case of Torus~6 for $G_2(q)$, it is true that $a^6=1$. Hence, $K=\langle a \rangle$ is a complement to $\overline{T}_{\sigma n}$ in $\overline{N}_{\sigma n}$.

\textbf{Torus~3.} In this case $w=w_3$. Using Lemma~\ref{coroots}, we obtain 
$$(h_{r_1}(t_1)h_{r_2}(t_2))^{\sigma n_3}=\left(h_{r_1}(t_2^{3^m})h_{r_2}(t_1^{3^{m+1}})\right)^{n_{r_1+r_2}}=
h_{2r_1+r_2}(t_2^{3^m})h_{-3r_1-2r_2}(t_1^{3^{m+1}})=$$
$$h_{r_1}(t_2^{2\cdot3^m})h_{r_2}(t_2^{3^{m+1}})h_{r_1}(t_1^{-3^{m+1}})h_{r_2}(t_1^{-2\cdot3^{m+1}})=
h_{r_1}(t_1^{-3^{m+1}}t_2^{2\cdot3^m})h_{r_2}(t_1^{-2\cdot3^{m+1}}t_2^{3^{m+1}}).$$
This implies that elements of $\overline{T}_{\sigma n_3}$ are given by the following two equalities on $t_1$ and $t_2$:
$$t_1=t_1^{-3^{m+1}}t_2^{2\cdot3^m},\quad t_2=t_1^{-2\cdot3^{m+1}}t_2^{3^{m+1}}.$$
Therefore, 
\begin{multline*}
\begin{cases}
t_1^{3^{m+1}+1}=t_2^{2\cdot 3^m} \\
t_1^{2\cdot 3^{m+1}}=t_2^{3^{m+1}-1}
\end{cases}\Leftrightarrow
\begin{cases}
t_1^{3^{m+1}+1}=t_2^{2\cdot 3^m} \\
t_1^{3^{m+1}-1}=t_2^{3^{m}-1}
\end{cases}
\Leftrightarrow
\begin{cases}
t_1^{2}=t_2^{3^m+1} \\
t_1^{3^{m+1}-1}=t_2^{3^{m}-1}
\end{cases}
\\ 
\Leftrightarrow
\begin{cases}
t_1^{2}=t_2^{3^m+1} \\
t_2^{\frac{3^{2m+1}+2\cdot3^m-1}{2}}=t_2^{3^{m}-1}
\end{cases}\Leftrightarrow
\begin{cases}
t_1^{2}=t_2^{3^m+1} \\
t_2^{\frac{3^{2m+1}+1}{2}}=1
\end{cases}.
\end{multline*}
Hence, $t_2^\frac{q+1}{2}=1$ and $t_1=\pm t_2^\frac{{\sqrt{q/3}+1}}{2}$.
It follows that  $\overline{T}_{\sigma n_3}$ has a cyclic structure $\mathbb{Z}_2\times\mathbb{Z}_{\frac{q+1}{2}}$,
where each element in the direct product can be written as follows: $(t_1,t_2)=(t_1\cdot t_2^\frac{-{\sqrt{q/3}-1}}{2},1)\cdot(t_2^\frac{{\sqrt{q/3}+1}}{2},t_2)$.

Note that $C_{W,\sigma}(w)\simeq\langle w_1w_2 \rangle\simeq\mathbb{Z}_6$.
Consider elements $a=n_1n_2$ and $n=h_2n_3$. Using MAGMA, we see that
$a^{\sigma n}=a.$ Hence $K=\langle a \rangle$ is a complement to $\overline{T}_{\sigma n}$ in $\overline{N}_{\sigma n}$.

\textbf{Torus~4.} In this case $w=w_4$.
Then
$$(h_{r_1}(t_1)h_{r_2}(t_2))^{\sigma n_4}=\left(h_{r_1}(t_2^{3^m})h_{r_2}(t_1^{3^{m+1}})\right)^{n_{2r_1+r_2}}=
h_{-r_1-r_2}(t_2^{3^m})h_{r_2}(t_1^{3^{m+1}})=$$
$$h_{r_1}(t_2^{-3^m})h_{r_2}(t_2^{-3^{m+1}})h_{r_2}(t_1^{3^{m+1}}).$$
We get the following equivalent systems of equations for the element parameters $t_1$, $t_2$ of $\overline{T}_{\sigma n_4}$:
$$\begin{cases}
t_1=t_2^{-3^m} \\
t_2=t_1^{3^{m+1}}t_2^{-3^{m+1}}
\end{cases}
\Leftrightarrow
\begin{cases}
t_1=t_2^{-3^m} \\
t_2=(t_2^{-3^m})^{3^{m+1}}t_2^{-3^{m+1}}
\end{cases}
\begin{cases}
t_1=t_2^{-3^m} \\
t_2^{3^{2m+1}+3^{m+1}+1}=1
\end{cases}.
$$
Therefore, the elements of the torus are parametrized by a set
$\{(z,z^{-\sqrt{3q}})~|~z\in\overline{\F}_p, z^{q+\sqrt{3q}+1}=1\}$, 
so the torus is isomorphic to the cyclic group of order $q+\sqrt{3q}+1$.

Note that $C_{W,\sigma}(w)\simeq\langle w_1w_2 \rangle\simeq\mathbb{Z}_6$.
Consider elements $a=n_1n_2$ and $n=h_2n_4$. Using MAGMA, we see that
$a^{\sigma n}=a.$ 
Therefore, $K=\langle a \rangle$ is a complement to $\overline{T}_{\sigma n}$ in $\overline{N}_{\sigma n}$.

\begin{table}[H]
\begin{center}
\caption{Normalizers of maximal tori of ${}^2G_2(q)$\label{table2G_2}}
{\centering
\begin{tabular}{|l|c|l|c|}
\hline
&  $w$ & Structure of $T$ & $C_{W,\sigma}(w)$ \\ \hline
1  & $1\sim(w_1w_2)^k$ &  $\{(z,z^{\sqrt{3q}})~|~z^{q-\sqrt{3q}+1}=1\}$ &  $\mathbb{Z}_2$  \\
   & & $T\simeq q-1$  &  \\
2  & $w_1\sim w_2$ & $\{(z,z^{\sqrt{3q}})~|~z^{q-1}=1\}$  & $\mathbb{Z}_6$ \\
   & & $T\simeq q-\sqrt{3q}+1$ &  \\
3  & $w_3\sim w_5$ & $\{(t_1,t_2)~|~t_1=\pm t_2^{(\sqrt{q/3}-1)/2}, t_2^{(q+1)/2}=1\}$ & $\mathbb{Z}_6$    \\
   & & $T\simeq \mathbb{Z}_2\times\frac{q+1}{2}$ & \\
4  & $w_4\sim w_6$ & $\{(z,z^{-\sqrt{3q}})~|~z^{q+\sqrt{3q}+1}=1\}$ & $\mathbb{Z}_6$  \\
   & & $T\simeq q+\sqrt{3q}+1$ &  \\
\hline
\end{tabular}}
\end{center}
\end{table}

\section{Proof of Theorem~\ref{th} for groups $^3D_4(q)$}
In this section, we suppose that $G={}^3D_4(q)$, where $q$ is a power of a prime $p$.
Since the case $p=2$ follows from~\cite[Remark 2.5]{GS3}, we assume that $q$ is odd. The extended Dynkin diagram of type $D_4$ has the following form:

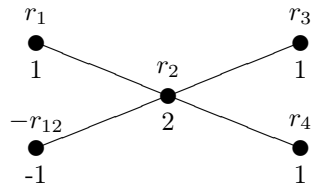
\begin{picture}(250,70)(-170,-30) \put(0,20){\circle*{6}}
    \put(0,20){\line(5,-2){50}} \put(50,0){\circle*{6}}
    \put(100,20){\circle*{6}}
    \put(100,-20){\circle*{6}} \put(50,0){\line(5,2){50}}
    \put(50,0){\line(5,-2){50}} \put(0,30){\makebox(0,0){$r_1$}}
    \put(0,10){\makebox(0,0){1}} \put(50,10){\makebox(0,0){$r_2$}}
    \put(100,30){\makebox(0,0){$r_3$}}
    \put(100,10){\makebox(0,0){1}} \put(0,-20){\circle*{6}}
    \put(0,-20){\line(5,2){50}} \put(0,-10){\makebox(0,0){$-r_{12}$}}
    \put(0,-30){\makebox(0,0){-1}}
    \put(50,-10){\makebox(0,0){2}} \put(100,-10){\makebox(0,0){$r_4$}}
    \put(100,-30){\makebox(0,0){1}}
\end{picture}\\

We denote by $\rho$ the symmetry of the Dynkin diagram such that $\rho(r_1)=r_3$, $\rho(r_2)=r_2$, $\rho(r_3)=r_4$, and $\rho(r_4)= r_1$.
We will use the following numbering of roots:
\begin{multline*}
r_5=r_1+r_2, r_6=r_2+r_3, r_7=r_2+r_4, r_8=r_1+r_2+r_3, r_9=r_1+r_2+r_4, \\r_{10}=r_2+r_3+r_4, r_{11}=r_1+r_2+r_3+r_4, r_{12}=r_1+2r_2+r_3+r_4.
\end{multline*}
According to~\cite[Theorems 2.2.3, 1.15.2, 1.15.4]{GorLySol},  the endomorphism $\sigma$ acts on generators of $\overline{G}$ in the following way:
$(x_r(t))^\sigma=x_{r^\rho}(t^q)$. In particular, 
$$n_r^\sigma=n_{r^\rho} \text{ and } (h_r(t))^\sigma=h_{r^\rho}(t^q).$$
A structure of maximal tori of $^3D_4(q)$ was given in~\cite[Proposition~1.2]{DerM} and is presented below in Table~\ref{table3D_4}.

It is well known that the Weyl group $W=\langle w_1,w_2,w_3,w_4\rangle$ of $D_4(q)$ is isomorphic to a subgroup of index $2$ in the group $2^4\rtimes \Sym_4$ and
contains a central involution $w_0=w_1w_3w_4w_{12}$.
The group $W$ has seven $\sigma$-conjugacy classes. We choose representatives 
for these classes according to Table~\ref{table3D_4}.

Calculations in MAGMA show that $n_0:=n_1n_3n_4n_{12}$ lies in $Z(\mathcal{T})$. In what follows, we use this fact without explanation.

\begin{table}[H]
    \begin{center}
    \caption{Normalizers of maximal tori of $^3D_4(q)$\label{table3D_4}}
    {\centering
    \begin{tabular}{|l|l|c|l|l|}
    \hline
&  $w$ & $|w|$ & $C_{W,\sigma}(w)$ & Structure of $T$ \\ \hline
1  & $1$ & 1 &   $D_{12}$ & $\{(t_1,t_2,t_1^q,t_1^{q^2})~|~t_1^{q^3-1}=t_2^{q-1}=1\}$ \\
   & & & & $T\simeq(q^3-1)\times(q-1)$ \\  \hline
2  & $w_{12}$ & 2 &  $\mathbb{Z}_2\times \mathbb{Z}_2$ & $\{(t,t^{1-q^3},t^{q^4},t^{q^2})~|~t^{(q^3-1)(q+1)}=1\}$ \\
   & & & & $T\simeq(q^3-1)(q+1)$ \\ \hline
3  & $w_0w_{12}$ & 2 &  $\mathbb{Z}_2\times \mathbb{Z}_2$ & $\{(t,t^{q^3+1},t^{q^4},t^{q^2})~|~t^{(q^3+1)(q-1)}=1\}$ \\ 
   & & & & $T\simeq(q^3+1)(q-1)$ \\ \hline
4  & $w_{12}w_2$ & 3 &  $\operatorname{SL}_2(3)$ & $\{(t_1,t_2,t_1^qt_2,(t_1^{-1}t_2)^{q+1})~|~t_i^{q^2+q+1}=1\}$  \\
   & & & & $T\simeq(q^2+q+1)\times(q^2+q+1)$  \\  \hline
5  & $w_0w_{12}w_2$ & 3 &  $\operatorname{SL}_2(3)$ & $\{(t_1,t_2,t_1^{-q}t_2,(t_1t_2^{-1})^{q-1})~|~t_i^{q^2-q+1}=1\}$  \\
   & & & & $T\simeq(q^2-q+1)\times(q^2-q+1)$  \\ \hline
6  & $w_1w_2$ & 3 &  $\mathbb{Z}_4$ & $\{(t,t^{q^3+1},t^q,t^{q^2})~|~t^{q^4-q^2+1}=1\}$  \\
   & & & & $T\simeq q^4-q^2+1$  \\ \hline
7  & $w_0$ & 2 &  $D_{12}$ & $\{(t_1,t_2,t_1^{-q},t_1^{q^2})~|~t_1^{q^3+1}=t_2^{q+1}=1\}$  \\ 
   & & & & $T\simeq(q^3+1)\times(q+1)$  \\ \hline
\end{tabular}}
\end{center}
\end{table}

For each $\sigma$-conjugacy class of maximal tori, we present a complement in the corresponding algebraic normalizer.

\textbf{Tori 1 and 7.} In this case $w=1$ or $w=w_0$, respectively.
Moreover, we see that $C_{W,\sigma}(w)\simeq\langle w_2, w_1w_3w_4 \rangle\simeq D_{12}$.

Define $n=1$ if $w=1$ and $n=n_0$ if $w=w_0$.
Consider elements $a=h_1h_3h_4n_2$ and $b=h_2n_1n_3n_4$.
Since 
$$n_2^{\sigma  n}=n_2^n=n_2,\quad (n_1n_3n_4)^{\sigma n}=n_3n_4n_1^n=n_1n_3n_4,$$
we infer that $n_2, n_1n_3n_4\in\overline{N}_{\sigma n}$.
According to Table~\ref{table3D_4}, it is true that $h_2$, $h_1h_3h_4\in\overline{T}_{\sigma n}$ and hence $a,b\in\overline{N}_{\sigma n}$. Using MAGMA, we find that $a^2=b^2=1$ and $(ab)^6=1$. This implies that $K=\langle a,b \rangle$ is a complement to $\overline{T}_{\sigma n}$ in $\overline{N}_{\sigma n}$.

\textbf{Tori 2 and 3.} In this case $w=w_{12}$ or $w=w_0w_{12}$, respectively.
Moreover, we see that $C_{W,\sigma}(w)\simeq\langle w_0,w_{12} \rangle\simeq \mathbb{Z}_2\times \mathbb{Z}_2$.

Define $n=n_{12}$ and $\varepsilon=-$ if $w=w_{12}$; $n=n_0n_{12}$ and $\varepsilon=+$ if $w=w_0w_{12}$.

Consider elements $\alpha, \beta\in\overline{\mathbb{F}}_p$ such that $\alpha^{\frac{((\varepsilon{q})^3+1)(\varepsilon{q}-1)}{2}}=-1$ and $\beta=\alpha^{\frac{((\varepsilon{q})^3+1)}{2}}$. Define elements
$$H_1=(\alpha, \alpha^{(\varepsilon{q})^3+1}, \alpha^{q^4}, \alpha^{q^2}),\quad H_2=(\beta, \beta^{(\varepsilon{q})^3+1}, \beta^{q^4}, \beta^{q^2}).$$
According to~\cite[Table~1.1]{DerM}, we conclude that $H_1, H_2\in\overline{T}_{\sigma n}$.

Consider elements $a=H_1n_0$ and $b=H_2n_{12}$. Since
$$n_0^{\sigma n}=n_0^n=n_0,\quad n_{12}^{\sigma n}=n_{12}^n=1,$$
we infer that $n_0, n_{12}\in\overline{N}_{\sigma n}$.
It follows from~\cite[Lemma~3.2]{GS2} that
$$(Hn_{12})^2=(-\lambda_1^2\lambda_2^{-1}, 1, -\lambda_3^2\lambda_2^{-1}, -\lambda_4^2\lambda_2^{-1}).$$
Since $\beta^{\varepsilon{q}}=-\beta$, we find that $\beta=\beta^{q^2}=\beta^{q^4}=-\beta^{(\varepsilon{q})^3}$. Therefore,
$$b^2=(H_2n_{12})^2=(-\beta^{1-(\varepsilon{q})^3}, 1, -\beta^{2q^4-(\varepsilon{q})^3-1}, -\beta^{2q^2-(\varepsilon{q})^3-1})=(1, 1, 1, 1)=1.$$
By \cite[Lemma~3.3]{GS2}, we see that
$$[a, b]=1 \Leftrightarrow H_1^{-1}H_1^{n_{12}}=H_2^{-1}H_2^{n_0}=H_2^{-2}.$$
Since 
$$H^{n_{12}}=(\lambda_1\lambda_2^{-1}, \lambda_2^{-1}, \lambda_3\lambda_2^{-1},  \lambda_4\lambda_2^{-1}),$$
it is true that
$$H_1^{-1}H_1^{n_{12}}=(\alpha^{-((\varepsilon{q})^3+1)}, \alpha^{-2((\varepsilon{q})^3+1)}, \alpha^{-((\varepsilon{q})^3+1)},  \alpha^{-((\varepsilon{q})^3+1)})=(\beta^{-2}, \beta^{-4}, \beta^{-2}, \beta^{-2}).$$
On the other hand, $\beta^{(\varepsilon{q})^3+1}=-\beta^2$ and
$$H_2^{-2}=(\beta^{-2}, \beta^{-2((\varepsilon{q})^3+1)}, \beta^{-2q^4}, \beta^{-2q^2})=
(\beta^{-2}, (-\beta^2)^{-2}, \beta^{-2}, \beta^{-2}).$$
Finally, for every $H$ it is true that $(Hn_0)^2=n_0^2=1$.
Therefore, $K=\langle a, b \rangle\simeq \mathbb{Z}_2\times \mathbb{Z}_2$ is a complement to $\overline{T}_{\sigma n}$ in $\overline{N}_{\sigma n}$.

\textbf{Tori 4 and 5.} In this case $w=w_{12}w_2$ or $w=w_0w_{12}w_2$, respectively.
Moreover, it is true that $C_{W,\sigma}(w)\simeq\operatorname{SL}_2(3)$.
According to GAP, the group $\operatorname{SL}_2(3)$ has the following presentation:
$$\operatorname{SL}_2(3)\simeq\langle a,b~|~a^4, b^3, aba^{-1}bab, (b^{-1}a)^3\rangle.$$
Moreover, elements
$w_1w_7$ and $w_1w_2w_3w_7$ generate $C_{W,\sigma}(w)$
and satisfy this set of relations.

Define $n=n_{12}n_2$ if $w=w_{12}w_2$ and $n=n_0n_{12}n_2$ if $w=w_0w_{12}w_2$.

Consider elements $a=n_1n_2n_3n_7$ and $b=n_1n_7$. 
Since
$$a^{\sigma n}=(n_3n_2n_4n_5)^n=a,\quad b^{\sigma n}=(n_3n_5)^n=b,$$
we infer that $a, b\in\overline{N}_{\sigma n}$.
Using MAGMA, we see that $a^4=b^3=aba^{-1}bab=(b^{-1}a)^3=1$. Therefore,  $K=\langle a,b \rangle$ is a complement to $\overline{T}_{\sigma n}$ in $\overline{N}_{\sigma n}$.

\textbf{Torus~6.} In this case $w=w_1w_2$ and $C_{W,\sigma}(w)\simeq\langle w_1w_2w_3w_7 \rangle\simeq \mathbb{Z}_4$.

Consider $n=n_1n_2$ and $a=n_1n_2n_3n_7$. Since
$a^{\sigma n}=(n_3n_2n_4n_5)^n=a,$
we infer that $a\in\overline{N}_{\sigma n}$.

Using MAGMA, we find that $a^4=1$ and hence
$K=\langle a\rangle$ is a complement to $\overline{T}_{\sigma n}$ in $\overline{N}_{\sigma n}$.

\section{Proof of Theorem~\ref{th_Omega}}

In this section, we suppose that $G=\mathrm{P}\Omega_{2n}^-(q)$, where $n\geqslant4$ and $q$ is a power of a prime $p$.
Notice that~\cite[Remark 2.5]{GS3} is true for all groups of Lie type and therefore a maximal torus $T$ has a complement in its algebraic normalizer $N(G,T)$ in the case of even characteristic. In what follows, we assume that $q$ is odd.

Recall some definitions and the notation concerning orthogonal groups from~\cite{ButGre}. 
The group $\GO_{2n+1}(\overline{\F}_p, Q)$ is an orthogonal group of dimension $2n+1$ over $\overline{\F}_p$ associated with a non-singular quadratic form  $Q$, where $Q(v)=x_0^2+x_1x_{-1}+\ldots+x_nx_{-n}$. We fix a basis $\{x,e_1,\ldots,e_n,f_1\ldots,f_n\}$ corresponding to $Q$ in the vector space $V$. We numerate rows and columns of matrices in $\GO_{2n+1}(\overline{\F}_p)$ in the order $0,1,2,\ldots,n,-1,-2,\ldots,-n$. Define $\overline{G}$ as a subgroup of $\overline{H}=\SO_{2n+1}(\overline{\F}_p)$ consisting of all matrices of the form $\bd(1, A)$, where $A$ is a matrix of size $2n\times 2n$. Then $\overline{G}\simeq\SO_{2n}(\overline{\F}_p)$.

A subgroup $\overline{T}$ of $\overline{H}$ consisting of all diagonal matrices of the form $\bd(1,D,D^{-1})$ is a maximal torus of groups $\overline{H}$ and $\overline{G}$.
The group $N_{\overline{H}}(\overline{T})$ is a subgroup of the monomial matrix group. There exists an embedding of a Weyl group $W_{\overline{H}}=N_{\overline{H}}(\overline{T})/\overline{T}$ into a permutation group on the set $\{1,2,\ldots,n,-1,-2,\ldots,-n\}$. The image of $W$ under this embedding coincides with the group $\Sl_n$ of all permutations $\varphi$ such that $\varphi(-i)=-\varphi(i)$. 

If we drop signs from elements of the set $\{1,2,\ldots,n,-1,-2,\ldots,-n\}$, we obtain a homomorphism from $\Sl_n$ onto $\Sym_n$. Let $\varphi\in\Sl_n$ is mapped into a cycle $(i_1i_2\ldots i_k)$ and fixes all elements, except $\pm i_1,\pm i_2,\ldots,\pm i_k$. If $\varphi(i_k)=i_1$, then $\varphi$ is called a positive cycle of length $k$; if $\varphi(i_k)=-i_1$, then $\varphi$ is called a negative cycle of length $k$. The image of an arbitrary element  $\varphi$ of $\Sl_n$ is uniquely expressed as a product of disjoint cycles. According to this decomposition, $\varphi$ is uniquely expressed as a product of disjoint positive and negative cycles. Lengths of the cycles together with their signs give a set of integers, called the {\it cycle type} of $\varphi$.

The Weyl group $W_{\overline{G}}=N_{\overline{G}}(\overline{T})/\overline{T}$ is isomorphic to a subgroup $\Sl_n^+$ of $\Sl_n$ consisting of all permutations whose decomposition into disjoint cycles has an even number of negative cycles.

Let $n_0=\bd(-1,A_0)$, where $A_0$ is the permutation matrix corresponding to the negative cycle $(n,-n)$. Then $n_0\in N_{\overline{H}}(\overline{T})$ and $W_{\overline{H}}=W_{\overline{G}}\cup w_0W_{\overline{G}}$, where $w_0=\pi(n_0)$.

Let $\sigma$ be a Steinberg endomorphism of $\overline{H}$, acting by the rule $(a_{ij})\mapsto(a_{ij}^q)$, and $\sigma_1=\sigma\circ n_0$. Then $G=\overline{G}_{\sigma_1}\simeq\SO_{2n}^-(q)$ and $O^{p'}(G)\simeq\Omega_{2n}^-(q)$. According to~\cite[\S4]{ButGre}, two elements $w_1$ and $w_2$ are $\sigma_1$-conjugate in $W_{\overline{G}}$ if and only if $w_0w_1$ and $w_0w_2$ are conjugate by an element of $W_{\overline{G}}$. In particular, 
$$C_{W_{\overline{G}},\sigma_1}(w)=C_{W_{\overline{G}},\sigma}(w_0w).$$
We say that a maximal torus $(\overline{T}^g)_{\sigma_1}$ of $G$ with $\pi(g^\sigma g^{-1})=w$ corresponds to the element $w_0w$.

For brevity, we assume $W=W_{\overline{G}}$ and $\overline{N}=N_{\overline{G}}(\overline{T})$. 
According to Proposition~\ref{prop2.5}, we get that
$$({\overline{T}}^g)_{\sigma_1}=(\overline{T}_{\sigma_1 n})^g=(\overline{T}_{\sigma n_0n})^g, \quad (N_{\overline{G}}({\overline{T}}^g))_{\sigma_1}=(\overline{N}^g)_{\sigma_1}=(\overline{N}_{\sigma n_0n})^g.$$ 
Hence,
$$C_{W,\sigma}(w_0w)= C_{W,\sigma_1}(w)\simeq  (N_{\overline{G}}({\overline{T}}^g))_{\sigma_1}/({\overline{T}}^g)_{\sigma_1}= (\overline{N}_{\sigma n_0n})^g/(\overline{T}_{\sigma n_0n})^g\simeq\overline{N}_{\sigma n_0n}/\overline{T}_{\sigma n_0n}.$$
Maximal tori of $\SO_{2n}^-(q)$ have the following description.
\begin{prop}{{\rm \cite{ButGre}}}\label{prop2}
Let $(\overline{T}^g)_{\sigma_1}$ be a maximal torus of $\overline{G}_{\sigma_1}=\SO_{2n}^-(q)$ corresponding to an element $w_0w$ of the Weyl group with the cycle type $(\overline{n_1})\ldots(\overline{n_k})(n_{k+1})\ldots(n_m)$, where $k$ is odd. Put $\varepsilon_i=-$ if $i\leqslant k$ and $\varepsilon_i=+$ otherwise. Then $\overline{T}_{\sigma_1 n}$ is a subgroup of $\overline{G}$ consisting of all
diagonal matrices of the form
$$\bd(1, D_1, D_2,\ldots,D_m, D_1^{-1}, D_2^{-1},\ldots, D_m^{-1}),$$ 
where $D_i=\diag(\lambda_i, \lambda_i^q,\ldots, \lambda_i^{q^{n_i-1}})$ and $\lambda_i^{q^{n_i}-\varepsilon_i1}=1$. 
\end{prop}
When $k$ is even, the subgroup $T$ from Proposition~\ref{prop2} is isomorphic to a maximal torus of $\overline{G}_{\sigma}=\SO_{2n}^+(q)$ and the results for this case were obtained in~\cite{Galt2}. In order to prove Theorem~\ref{th_Omega} we will refer to the mentioned paper.

First, we note that there is an error in~\cite[Proposition~5]{Galt2}. Since the quadratic form $Q$ was already fixed, we have $Q(x)=1$ (instead of $Q(x)=-1$ as in~\cite[Proposition~5]{Galt2}) and the statement should be formulated as follows.

\begin{prop}\label{Omega} 
Let $g\in\SO_{2n+1}(q)$ such that
$$g(e_j)=\alpha f_j,\; g(f_j)=\alpha^{-1}e_j,\; g(x)=-x,\; g(e_i)=e_i,\; g(f_i)=f_i,$$
for $1\leqslant i \leqslant n,\;i\neq j, \alpha\in\F_q^*$. Then $g\in\Omega_{2n+1}(q)$ if and only if $(-\alpha)\in(\F_q^*)^2$.
\end{prop}
Because of this, some proofs of Lemmas in~\cite{Galt2} must be corrected, but the results, stated in the theorems~2~and~3 of~\cite{Galt2}, remain valid. 

Recall the notation used in~\cite{Galt2} to describe $C_W(w)$.

\noindent $\omega_1=(1,2,\ldots,n_1)(-1,-2,\ldots,-n_1)$;\\
$\varpi_1=(1,2,\ldots,n_1,-1,-2,\ldots,-n_1)$;\\
$\tau_1=(1,-1)(2,-2)\ldots(n_1,-n_1)$;\\
$\chi_1=(1,n_1+1)(2,n_1+2)\ldots(n_1,2n_1)(-1,-(n_1+1))(-2,-(n_1+2))\ldots(-n_1,-2n_1)$.

Let $t=n_1+\ldots+n_i$. Similarly, we define the elements\\
$\omega_{i+1}=(t+1,\ldots,t+n_{i+1})(-(t+1),\ldots,-(t+n_{i+1}))$;\\
$\varpi_{i+1}=(t+1,\ldots,t+n_{i+1},-(t+1),\ldots,-(t+n_{i+1}))$;\\
$\tau_{i+1}=(t+1,-(t+1))\ldots(t+n_{i+1},-(t+n_{i+1}))$;\\
$\chi_{i+1}=(t+1,t+n_{i+1}+1)\ldots(t+n_{i+1},t+2n_{i+1})
(-(t+1),-(t+n_{i+1}+1))\ldots(-(t+n_{i+1}),-(t+2n_{i+1}))$.\\

(a) If an element $w$ has the cycle type $(\overline{n_1})\ldots(\overline{n_k})$ and $n_1=n_2=\ldots=n_k$, then
$$C_W(w)=\langle\varpi_i,\chi_j|1\leqslant i\leqslant k, 1\leqslant j\leqslant k-1\rangle,\quad C_W(w)\simeq\mathbb{Z}_{2n_{1}}\wr\Sym_{k}.$$

(b) If an element $w$ has the cycle type $({n_1})\ldots({n_k})$ and $n_1=n_2=\ldots=n_k$, then
$$C_W(w)=\langle\omega_i,\tau_i,\chi_j|1\leqslant i\leqslant k, 1\leqslant j\leqslant k-1\rangle,\quad C_W(w)\simeq(\mathbb{Z}_{n_1}\times\mathbb{Z}_{2})\wr\Sym_{k}.$$
In the general case, the group $C_W(w)$ is isomorphic to the direct product of the groups from (a) and (b).

Now we give the required corrections for the work~\cite{Galt2}.
According to Proposition~\ref{Omega}, the spinor norm $\theta$ of the element $$\tau=(1,-1)(2,-2)\ldots(k,-k)$$ 
is equal to $\theta(\tau)=(-1)^k$.

In the proof of Lemma~13 from~\cite{Galt2} we should define the elements $u_i$ as follows
\[{u}_i=\begin{cases} \bd((-1)^{n_i},I_n,I_n)\tau_i & \text{if } n_i \text{ is even} \\
	v_0\bd((-1)^{n_i},I_n,I_n)\tau_i & \text{if } n_i \text{ is odd}
\end{cases}.
\]
Then the spinor norm $\theta(u_i)=1$ for all $n_i$. All the rest in the proof of~\cite[Lemma~13]{Galt2} remains the same.

In the proof of Lemma~18 from~\cite{Galt2} we should define the elements $u_i$ as follows
$$u_1=\bd(1,T_1,T_1^{-1})\tau_1, u_2=\bd(1,T_2,T_2^{-1})\tau_2, u_3=\bd(1,T_3,T_3^{-1})\tau_3, u_4=\bd(1,T_4,T_4^{-1})\tau_4,$$
where $$T_1=\bd(-I_1,I_1,I_3,I_3), T_2=\bd(I_1,-I_1,I_3,I_3), T_3=\bd(I_1,I_1,-I_3,I_3), T_4=\bd(I_1,I_1,I_3,-I_3),$$
and $I_j$ is the identity $n_j\times n_j$ matrix. Then the spinor norm $\theta(u_i)=\det(T_i)(-1)^{n_i}=1$ for all $i$. All the rest in the proof of~\cite[Lemma~18]{Galt2} remains the same.

In the proof of Lemma~23 from~\cite{Galt2} we should define the elements $t_1, t_2$ as follows
$$t_1=T_1\varpi_1, t_2=T_2\varpi_2, \text{ where} \quad T_1=T_2=\bd(-1,-I_1, I_1, -I_1, I_1).$$
Then the spinor norm $\theta(t_i)=\det(-I_1)\theta(\varpi_i)=(-1)^{n_i}(-1)^{n_i}=1$ for $i=1,2$. All the rest in the proof of~\cite[Lemma~23]{Galt2} remains the same.

\vspace{0.5em} 
Now we proceed to the proof of Theorem~\ref{th_Omega}. 

Lemmas~15, 16, and 19 of~\cite{Galt2} proved for even $k$ are also valid for odd $k$ because the parity is not used in the proofs. Thus, Lemmas~15, 16, and 19 imply the theorem for $m\geqslant5$, as well as for $q\equiv1\!\pmod4$.

Let $m=4$. In this case an element of the Weyl group has one of two cycle types: 
$$(\overline{n_1})(n_2)(n_3)(n_4) \text{ or } (\overline{n_1})(\overline{n_2})(\overline{n_3})(n_4).$$
If all $n_i$ are even, then the result follows from~\cite[Lemma~19]{Galt2}. If the number of odd numbers among $n_1,n_2,n_3,n_4$ is not equal to two, then the result follows from~\cite[Lemma~15(2)]{Galt2}. In the case when there are two different odd numbers among $n_1,n_2,n_3,n_4$, the result also follows from~\cite[Lemma~15(2)]{Galt2}.

\begin{lemma}\label{m=4} 
Let $T$ be a maximal torus of $G=\Omega_{2n}^-(q)$ corresponding to an element $w_0w$ of the Weyl group with the cycle type $(\overline{n_1})(n_2)(n_3)(n_4)$ or $(\overline{n_1})(\overline{n_2})(\overline{n_3})(n_4)$. Let $\widetilde{T}$ and $\widetilde{N}$ be the images of $T$ and $N(G,T)$ in $\widetilde{G}=\mathrm{P}\Omega_{2n}^-(q)$, respectively. If $q\equiv3\pmod4$, $n_1,n_4$ are even, $n_2=n_3$ is odd, then $\widetilde{T}$ does not have a complement in~$\widetilde{N}$.
\end{lemma}
\begin{proof}
Assume on the contrary that $\widetilde{T}$ has a complement $\widetilde{H}$ in $\widetilde{N}$.
Let $H$ be a preimage of $\widetilde{H}$ in $N(G,T)$.
Since $n_2=n_3$, we have $\chi_2\in C_W(w_0w)$. Since $n_1$ and $n_4$ 
are even, we get that $\tau_1,\tau_4\in C_W(w_0w)$.
Let $v_2,u_1,u_4$ be preimages of $\chi_2,\tau_1,\tau_4$ in $H$. Then the element $v_2$ has the form
$$v_2=\bd(D_1,D_2,D_3,D_4,D_1^{-1},D_2^{-1},D_3^{-1},D_4^{-1})\chi_2,$$ where
$D_i=\diag(\lambda_i, \lambda_i^q,\ldots, \lambda_i^{q^{n_i-1}})$.
Since $\widetilde{H}$ is a complement to $\widetilde{T}$, the following equalities must be hold
$$v_2^{2}=\varepsilon I,\quad v_2u_1=\varepsilon_1 u_1v_2,\quad v_2u_4=\varepsilon_4 u_4v_2,  \text{ where } \varepsilon,\varepsilon_1,\varepsilon_4\in\{1,-1\}.$$
Since $m=4$, we find that $\varepsilon_1=\varepsilon_4=1$ and hence $\lambda_1^2=\lambda_4^2=1$ by~\cite[Lemma~4]{Galt2}.
Applying~\cite[Lemma~8]{Galt2} to the equality $v_2^{2}=\varepsilon I$, we get that $D_2D_3=\varepsilon I_1$, $\lambda_1^2=\lambda_4^2=\varepsilon$. Hence, $\varepsilon=1$ and
$$\det(v_2|_{V_0})=\det(D_1D_2D_3D_4)\det(\chi_2|_{V_0})=\det(D_2D_3)\lambda_1^{n_1}\lambda_4^{n_4}(-1)^{n_2}= -1.$$ Since $q\equiv3\pmod4$, we have $\det(v_2|_{V_0})\notin(\F_q^*)^2$ and $v_2\notin\Omega_{2n}^-(q)$; a contradiction.
\end{proof}

\noindent 
Let $m=3$. In this case an element of the Weyl group has one of two cycle types:  
$$(n_1)(n_2)(\overline{n_3}) \text{ or } (\overline{n_1})(\overline{n_2})(\overline{n_3}).$$
According to~\cite[Lemma~19]{Galt2} and~\cite[Lemma~15(2)]{Galt2}, it is sufficient to consider the case $n_1=n_2$ is odd and $n_3$ is even.

\begin{lemma}\label{m=3} 
Let $T$ be a maximal torus of $G=\Omega_{2n}^-(q)$ corresponding to an element $w_0w$ of the Weyl group with the cycle type $(n_1)(n_2)(\overline{n_3})$ or $(\overline{n_1})(\overline{n_2})(n_3)$. Let $\widetilde{T}$ and $\widetilde{N}$ be the images of $T$ and $N(G,T)$ in $\widetilde{G}=\mathrm{P}\Omega_{2n}^-(q)$, respectively. If $q\equiv3\pmod4$, $n_1=n_2$ is odd, and $n_3$ is even, then $\widetilde{T}$ does not have a complement in~$\widetilde{N}$.
\end{lemma}
\begin{proof}
Assume on the contrary that $\widetilde{T}$ has a complement $\widetilde{H}$ in $\widetilde{N}$.
Let $H$ be a preimage of $\widetilde{H}$ in $N(G,T)$. Since $\varpi_3,\tau_1\notin W$, we have $\varpi_3\tau_1\in W$ and
$$\langle\chi_1,\varpi_3\tau_1\rangle\leqslant C_W(w_0w).$$
Let $v$ and $t$ be preimages of $\chi_1$ and $\varpi_3\tau_1$ in $H$, respectively.
Then the elements $v,t$ have the form
\begin{center}
$v=\bd(D_1,D_2,D_3,D_1^{-1},D_2^{-1},D_3^{-1})\chi_1, \quad D_i=\diag(\lambda_i, \lambda_i^q, \ldots, \lambda_i^{q^{n_i-1}}),$
\end{center}
\begin{center}
$t=\bd(U_1,U_2,U_3,U_1^{-1},U_2^{-1},U_3^{-1})\varpi_3\tau_1, \quad U_i=\diag(\mu_i, \mu_i^q, \ldots, \mu_i^{q^{n_i-1}})$,
\end{center}
Since $\widetilde{H}$ is a complement to $\widetilde{T}$, we infer that the
following equalities must be hold
$$v^{2}=\varepsilon I, vt=\delta tv, \quad \text{ where } \varepsilon,\delta\in\{1,-1\}.$$
Applying~\cite[Lemma~8]{Galt2} to the equality $v^{2}=\varepsilon I$, we get that $D_1D_2=\varepsilon I_1$ and $\lambda_3^2=\varepsilon$.

Consider the equality $vt=\delta tv$. If $\delta=-1$, then by~\cite[Lemma 3(2)]{Galt2} we get that $n_3$ is odd; a contradiction. Thus,
$\delta=1$ and it follows from \cite[Lemma 3(1)]{Galt2} that $\lambda_3^2=1$. 

Hence, $\varepsilon=1$ and $D_1D_2=1$. Since $q\equiv3\pmod4$, we have
$$\theta(v)=\det(D_1D_2D_3)\theta(\chi_1)=\lambda_3^{n_3}(-1)^{n_1}=-1\notin(\F_q^*)^2$$
and $v\notin\Omega_{2n}^-(q)$, a contradiction.
\end{proof}

When $m=2$, taking into account~\cite[Lemma~15(2)]{Galt2}, it remains to consider the following case.  
\begin{lemma}\label{m=2} 
Let $T$ be a maximal torus of $G=\Omega_{2n}^-(q)$ corresponding to an element $w_0w$ of the Weyl group with the cycle type  
$(\overline{n_1})(n_2)$. Let $\widetilde{T}$ and $\widetilde{N}$ be the images of $T$ and $N(G,T)$ in $\widetilde{G}=\mathrm{P}\Omega_{2n}^-(q)$, respectively. If $q\equiv3\pmod4$ and $n_1=n_2$ is even, then $\widetilde{T}$ does not have a complement in~$\widetilde{N}$.
\end{lemma}
\begin{proof}
Assume on the contrary that $\widetilde{T}$ has a complement $\widetilde{H}$ in $\widetilde{N}$.
Let $H$ be a preimage of $\widetilde{H}$ in $N(G,T)$. Since
$n_1,n_2$ are even, we infer that $\tau_1,\tau_2\in C_W(w_0w)$ and $C_W(w_0w)\geqslant\langle\omega_2,\tau_1,\tau_2\rangle$. 
Now we are in the conditions of the proof of~\cite[Lemma~21]{Galt2} where a contradiction is obtained.
\end{proof}
The remaining case $m=1$ follows from~\cite[Lemma~15(3)]{Galt2}.
Thus, all cases for $m$ are considered and Theorem~\ref{th_Omega} is proved.

\section{Proof of Theorem~\ref{final}}

For groups of Lie type $A_n(q)$ and $C_n(q)$ the theorem follows from~\cite{Galt3} and~\cite{Galt4}, respectively. Groups $B_n(q)$ and $D_n(q)$ were considered in~\cite{Galt2}. The results for twisted classical groups ${}^2A_n(q)$ and ${}^2D_n(q)$ are presented in Theorems~\ref{th_PSU} and~\ref{th_Omega}.

Finite groups of Lie type $E_6(q)$ were considered in~\cite{GS}, and groups $E_7(q)$, $E_8(q)$, and ${}^2E_6(q)$ were considered in~\cite{GS2}. The paper~\cite{GS3} is devoted to exceptional groups $F_4(q)$. In Theorem~\ref{th} we consider groups $G_2(q)$, ${}^2G_2(q)$, and ${}^3D_4(q)$.

Finally, suppose that $G={}^2F_4(q)$ or $G={}^2B_2(q)$. Let $T$ be a maximal torus of $G$ and $N$ be its algebraic normalizer. Since ${}^2F_4(q)$ and ${}^2B_2(q)$ are defined over a field of characteristic $2$, it follows from~\cite[Remark 2.5]{GS3} that $N$ splits over $T$.
\vspace{1em}

\noindent The authors are grateful the referee for valuable comments and suggestions on this work.

\Addresses

\begin{thebibliography}{99}

\bibitem{Tits}
J.\,Tits, {\it Normalisateurs de tores I. Groupes de Coxeter \'{E}tendus}, J. Algebra, \textbf{4} (1966), 96--116.

\bibitem{MalleTester}
G.\,Malle, D.\,Testerman, {\it Linear Algebraic Groups and Finite Groups of Lie Type}, Cambridge University Press, 2011.

\bibitem{AdamsHe}
J.\,Adams, X.\,He, {\it Lifting of elements of Weyl groups}, J. Algebra, \textbf{485} (2017), 142--165.

\bibitem{Galt1}
A.A.\,Galt, {\it On the splitting of the normalizer of a maximal torus in the exceptional linear algebraic groups}, Izv. Math., \textbf{81}:2 (2017), 269--285.

\bibitem{Galt2}
A.A.\,Galt, {\it On splitting of the normalizer of a maximal torus in orthogonal groups}, J. Algebra Appl., \textbf{16}:9 (2017), 1750174, 23 pp.

\bibitem{Galt3}
A.A.\,Galt, {\it On splitting of the normalizer of a maximal torus in linear groups}, J. Algebra Appl., \textbf{14}:7 (2015), 1550114, 20 pp.

\bibitem{Galt4}
A.A.\,Galt, {\it On the splitting of the normalizer of a maximal torus in symplectic groups},
 Izv. Math., \textbf{78}:3 (2014), 443--458.


\bibitem{LieGroups} M.\,Curtis, A.\,Wiederhold, B.\,Williams, {\it Normalizers of maximal tori},  Lecture Notes in Math., Springer, Berlin, \textbf{418} (1974), 31--47.

\bibitem{GS}
A.A.\,Galt, A.M.\,Staroletov, {\it On splitting of the normalizer of a maximal torus in $E_6(q)$}, Algebra Colloq., \textbf{26}:2 (2019), 329--350.

\bibitem{GS2}
A.A.\,Galt, A.M.\,Staroletov, {\it On splitting of the mormalizer of a maximal torus in $E_7(q)$ and $E_8(q)$}, Siberian Adv. Math., \textbf{31}:4 (2021), 244--282.

\bibitem{GS3}
A.A.\,Galt, A.M.\,Staroletov, {\it Minimal supplements of maximal tori in their normalizers for the groups $F_4(q)$}, Izv. Math., \textbf{86}:1 (2022), 126--149.

\bibitem{ButGre}
A.A.\,Buturlakin, M.A.\,Grechkoseeva, {\it The cyclic structure of maximal tori in finite classical groups}, Algebra Logic, \textbf{46}:2 (2007), 73--89.

%\bibitem{Atlas} J.H.\,Conway, R.T.\,Curtis, S.P.\,Norton, R.A.\,Parker, R.A.\,Wilson,{\it Atlas of Finite Groups}, Clarendon Press, Oxford, 1985.

\bibitem{CarSG}
R.W.\,Carter, {\it Simple groups of Lie type}, John Wiley and Sons,~1972.

\bibitem{GorLySol} D.\,Gorenstein, R.\,Lyons, R.\,Solomon,
{\slshape The classification of the finite simple groups. Number 3. Part I.
Chapter A. Almost simple $K$-groups}, Mathematical Surveys and Monographs,
{\bfseries 40}, N 3, American Mathematical Society, Providence, RI, 1998.

\bibitem{Car}
R.W.\,Carter, {\it Finite groups of Lie type, Conjugacy classes and complex
characters}, John Wiley and Sons, 1985.

\bibitem{Bour} N. Bourbaki, {\it Lie Groups and Lie Algebras}, chapters 4--6, Elements of Mathematics (Berlin), Springer, Berlin, 2002.

\bibitem{Kantor}
W.M.\,Kantor, A.\,Seress, {\it Prime power graphs for groups of Lie type}, J. Algebra, \textbf{247} (2002), 370--434.

\bibitem{MAGMA} W.\,Bosma, J.\,Cannon, and C.\,Playoust, {\it The Magma algebra system. I. The user language}, J. Symbolic Comput., \textbf{24} (1997), 235--265.

\bibitem{MC} \url{http://magma.maths.usyd.edu.au/calc}.

\bibitem{GAP} {\it The GAP Group, GAP -- Groups, Algorithms, and Programming}, Version 4.10.1 (2019) (\url{https://www.gap-system.org}).

\bibitem{github}
\url{https://github.com/AlexeyStaroletov/GroupsOfLieType}.

%\bibitem{Law}
%R.\,Lawther, {\it The action of $F_4(q)$ on cosets of $B_4(q)$}, J. Algebra, \textbf{212} (1999), 79--118.

\bibitem{DerM}
D.I.\,Deriziotis, G.O.\,Michler G.O., {\it Character table and blocks of finite simple triality groups ${}^3D_4(q)$}, Trans. Amer. Math. Soc., \textbf{303}:1 (1987), 39--70.


\end{thebibliography}
\end{document}